\documentclass[a4paper,reqno,10pt]{amsart}

\usepackage[utf8]{inputenc} 
\usepackage[T1]{fontenc} 
\usepackage[english]{babel} 
\usepackage{amsmath} 
\usepackage{amsthm}
\usepackage{amssymb}

\usepackage{nicefrac} 
\usepackage{hyperref}
\usepackage[margin=30mm]{geometry}
 \usepackage{mathrsfs}  
\usepackage{graphicx}
\usepackage{braket}
\usepackage{subfig}
\usepackage{csquotes}
\usepackage{comment}
\usepackage{xcolor}

\newcommand{\abs}[1]{\left|#1\right|}
\newcommand{\norm}[1]{\left \| #1\right \|}
\newcommand{\R}{\mathbb{R}}
\newcommand{\eps}{\epsilon }
\newcommand{\ue}{U_\delta}

\newcommand{\ud}{U_{\delta, \xi}}
\newcommand{\vd}{V_{\delta, \xi}}
\newcommand{\phid}{\Phi_{\delta, \xi}}
\newcommand{\psid}{\Psi_{\delta, \xi}}

\theoremstyle{plain}
\newtheorem{theorem}{Theorem}[]

\theoremstyle{plain}

\theoremstyle{plain}
\newtheorem{lemma}{Lemma}[]

\theoremstyle{plain}
\newtheorem{corollary}{Corollary}[]

\theoremstyle{plain}
\newtheorem{proposition}{Proposition}[]

\theoremstyle{remark}
\newtheorem{remark}{Remark}[]

\title[On a critical Hamiltonian system  with Neumann boundary conditions]{On a critical Hamiltonian system  with Neumann boundary conditions}
\author{Angela Pistoia and Delia Schiera}
\date{}

\begin{document}
\maketitle
\begin{abstract}
We consider the Hamiltonian  system with Neumann boundary conditions: 
\[ -\Delta u + \mu u=v^{q }, \quad 
-\Delta v+ \mu v=u^{p} \quad \text{ in $\Omega$}, \qquad 
 u, v >0 \quad  \text{ in $\Omega$,}
\qquad \partial_\nu u= \partial_\nu v=0 \quad  \text{ on $\partial \Omega$, }
\]
where $\mu >0$ is a parameter and $\Omega$ is a smooth bounded domain in $\mathbb R^N .$
 When $(p, q)$ approaches from below  the critical hyperbola $N/(p+1) + N/(q+1)=N-2$, we  build  a  solution
which blows-up at a boundary point where the mean curvature achieves its minimum and negative value. 
\end{abstract}


\section{Introduction and main result} 

This paper deals with the Hamiltonian system
\begin{equation}\label{s}
\begin{cases}
-\Delta u + \mu u= v^q & \text{ in $\Omega$}\\
-\Delta v+ \mu v=u^p & \text{ in $\Omega$}\\
  u, v >0 & \text{ in $\Omega$}\\
\partial_\nu u= \partial_\nu v=0 & \text{ on $\partial \Omega$.}
\end{cases}
\end{equation}
where $\Omega$ is a smooth bounded domain in $\mathbb R^n$, $p,q>1$ and $\mu>0$ is a positive parameter.
\\

The study of the system \eqref{s} is inspired by the widely studied equation
\begin{equation}\label{equ}
-\Delta u + \mu u= u^p\   \hbox{in}\ \Omega,\quad \partial_\nu u=0 \ \hbox{on}\ \partial \Omega.
\end{equation}

A huge number of results concerning existence and multiplicity of solutions of \eqref{equ} has been found starting from the pioneering paper \cite{lnt}.
Let us make a brief summary of the main results.
First of all, it is immediate to check that \eqref{equ} always possesses a constant solution for any $p>1$.\\
In the sub-critical regime, i.e. $p\in (1,+\infty)$ if $n=1,2$ or $p\in(1,\frac{n+2}{n-2})$ if $n\ge3$, there also exists a least energy solution $u_\mu$ which is not constant 
when the parameter $\mu$ is large enough as established in the serie  of papers \cite{lnt,nt1,nt2}. In particular, it has been proved that such a solution   exhibits a concentration phenomena   around  a boundary point, namely
$$u_\mu(x)\sim \mu^{1\over p-1}U\left(\mu^\frac12(x-x_0)\right)\ \hbox{as}\ \mu\to\infty$$
where $U$ is the unique positive radial solution to 
$$-\Delta U+U=U^p\ \hbox{in}\ \mathbb R^n$$
and the point $x_0\in\partial\Omega$ is the maximum point of  the mean curvature of $\partial \Omega.$  Successively, higher energy solutions    having such an asymptotic profile near
one or several points lying on  the boundary  $\partial\Omega$  or in  the interior of $\Omega$ have been built
 (see for instance \cite{dy,dfw,gpw,gw,lnw} and references
therein). Finally, it is worthwhile to point out that when the parameter $\mu$ is small enough the constant solution is the unique one as showed  in \cite{lnt}.
\\
In the critical regime, i.e. $p= {n+2\over n-2}$ and $n\ge3$, the existence of a least energy solution $u_\mu$ when $\mu$ is large enough has been established in \cite{am,w} and its asymptotic behaviour as $\mu\to\infty$ has been studied in  \cite{apy,npt,r}. In particular, it was proved that  $u_\mu$ blows-up at a boundary point, namely
$$u_\mu(x)\sim \delta^{-{n-2\over 2}}U\left( x-x_0\over\delta\right)\ \hbox{as}\ \mu\to\infty$$
where the bubble $U(x)=\alpha_n\frac1{(1+|x|^2)^{n-2\over2}}$, $\alpha_n=(n(n-2))^{\frac{n-2}4}$ solves  
$$-\Delta U =U^{n+2\over n-2}\ \hbox{in}\ \mathbb R^n,$$
the point $x_0\in\partial\Omega$ is the maximum point of  the mean curvature of $\partial \Omega $  and the positive parameter $\delta=\delta(\mu)\to0$ as $\mu\to\infty.$
Successively, solutions with this type of bubbling behavior around one or more boundary points 
have been built in a serie of papers \cite{am2,amy,g,gg,r2,wzi}.
Finally, we wish to point out that Lin and Ni in \cite{linni} conjectured that it was possible to extend to the critical regime the uniqueness result proved in the sub-critical one   when $\mu$ is small enough, i.e.  the constant solution is the unique solution to \eqref{equ} when $\mu\to0.$
However, the situation in the critical regime is completely different and the uniqueness of the solution is not true anymore, being closely related to   the geometry of the domain and also on the dimension $n$ as showed in   \cite{ay1,ay2,ay3,bkp,drw,wwy1,wwy2}.
\\
Finally, we wish to point out that
whereas all the previous results, concerned with
peaked solutions, always assume that $\mu$ goes to infinity, it is  also possible to prove that a blowing-up 
 solution of problem \eqref{equ} may exist for fixed $\mu$, provided that  the exponent $p$ is close to
the critical one, i.e.   $p={n+2\over n-2}\pm\epsilon$ and $\epsilon>0$ is a small parameter.
In particular, solutions blowing-up at {\em stable} critical points (including strict maxima and minima)  of the mean curvature have been built   in the slightly supercritical or  subcritical  regime if the  mean curvature critical value is positive or  negative, respectively (see 	\cite{dmp,rewe}).\\

Now, let us go back to the system \eqref{s}. 
We observe that system \eqref{s}   has a constant solution whatever $p >1$ and $q>1$ are. 
The role of the critical exponent $\frac{N+2}{N-2}$ is played by the   critical \textit{hyperbola} in the plane $p, q$: 
\begin{equation}\label{CH} \frac{1}{p+1} + \frac{1}{q+1} = \frac{N-2}{N}, \end{equation}
see \cite{BdST} for a survey on Lane--Emden systems with Dirichlet boundary conditions, and references therein.
 When $(p,q)$ lies below the critical hyperbola
the system has a variational structure. However,  the  functional naturally associated to \eqref{s}
  is strongly indefinite and finding a non-costant solution is not a simple matter. 
  In the sub-critical regime, i.e. 
$$ \frac{1}{p+1} + \frac{1}{q+1}> \frac{N-2}{N},$$
collecting the results obtained in  \cite{ay,ry,pr} we can claim that if $\mu$ is large enough the system \eqref{s} has a  solution $(u_\mu,v_\mu)$ 
which  exhibits  a   concentration phenomena  around a boundary  point, namely
$$u_\mu(x)\sim \mu^{1\over q-1}U\left(\mu^\frac12(x-x_0)\right),\ v_\mu(x)\sim \mu^{1\over p-1}V\left(\mu^\frac12(x-x_0)\right)\ \hbox{as}\ \mu\to\infty$$
where $(U,V)$ is the unique positive radial solution to the system
\begin{equation}\label{subsys}-\Delta U+U=V^q,\ -\Delta V+V=U^q\ \hbox{in}\ \mathbb R^n\end{equation}
and the point $x_0\in\partial\Omega$ is the maximum point of  the mean curvature of $\partial \Omega.$ \\

As far as we know, there are a few results concerning the critical case,
i.e. $$ \frac{1}{p+1} + \frac{1}{q+1}=\frac{N-2}{N}.$$
When $\mu=0$  a    solution (which has to change sign in this case) to \eqref{s} has been found     in \cite{ST2} in the subcritical case,  	in \cite{PST} in the critical case and in \cite{guopeng} in the slightly supercritical case.  When $\mu>0$  we refer the reader to \cite{guowy} for  the existence of infinitely many solutions in a symmetric domain through a Lyapunov-Schmidt approach and to   \cite{BST} where the authors find a pair of nondecreasing (and non-costant) radial solutions in the ball.
The existence of non-costant  solutions  when the parameter $\mu$ is large enough is still an interesting open problem.
\\

If we look at the results obtained for the single equation, many natural question arises. Here we list some of them which could be interesting for future research.  
{\it \begin{itemize}
\item[(Q1)] Do there exist solutions concentrating at  different boundary and/or interior points in the sub-critical    regime when $\mu\to\infty$?
\item[(Q2)] Do there exist solutions blowing-up at a boundary point  in the critical regime when $\mu\to\infty$?
\item[(Q3)] Do there exist solutions blowing-up at a boundary point   when $\mu>0$ is fixed and $(p,q)$ approaches the critical hyperbola?
\item[(Q4)] Is the constant the unique solution to \eqref{s} when $\mu$ is small in the sub-critical and/or the critical regime?
\end{itemize}}

In this paper we address question (Q3)  and we prove that   the  almost critical system,
\begin{equation}\label{system}
\begin{cases}
-\Delta u + \mu u=v^{q-\epsilon } & \text{ in $\Omega$}\\
-\Delta v+ \mu v=u^{p-\epsilon } & \text{ in $\Omega$}\\
\, u, v >0 & \text{ in $\Omega$}\\
\partial_\nu u= \partial_\nu v=0 & \text{ on $\partial \Omega$,}
\end{cases}
\end{equation}
when $(p, q)$ belongs to the critical hyperbola \eqref{CH}, possesses a positive solution which blows-up at a boundary point   as the positive parameter $\epsilon \to0$.
 \\
 Before stating our main result, it is necessary to introduce the bubbles which rule the bubbling behaviour in the critical case.
 It is known that the critical system
\begin{equation}\label{crisys}\begin{cases}
-\Delta U=V^q & \text{ in $\R^N$} \\
-\Delta V=U^p & \text{ in $\R^N$,}
\end{cases} \end{equation}
has a unique positive radial solution $(U,V)$ with  $U(0)=1$  and infinitely many positive solutions $\left(U_{\delta,\xi},V_{\delta,\xi}\right)$
obtained as 
\begin{equation}\label{bubbles}U_{\delta,\xi}(x):=\delta^{-{N \over q+1}} U \left(x-\xi\over\delta\right),\ V_{\delta,\xi}(x):=\delta^{-{N \over p+1}} V\left(x-\xi\over\delta\right),\ \delta>0,\ \xi\in\mathbb R^n.
\end{equation}
Moreover, we also introduce the function space 
\begin{equation}\label{xxx} X:=W_\nu^{2, \frac{q+1}{q}}(\Omega) \times  W_\nu^{2, \frac{p+1}{p}}(\Omega) \end{equation}
where  
\[ W^{2, s}_\nu(\Omega):=\{u \in W^{2, s}(\Omega): \partial _\nu u=0 \text{ on } \partial \Omega \} \]
which is  endowed with the norm
\[ \norm{(u, v)}_X:= \norm{-\Delta u + \mu u}_{ \frac{q+1}{q}} + \norm{-\Delta v + \mu v}_{\frac{p+1}{p}}. \]
Here and in the following we set $\|\cdot\|_s:=\|\cdot\|_{L^s(\Omega)}.$
 Finally, we introduce the projection $Pu\in W^{2, s}_\nu(\Omega)$ of a function $u\in W^{2,s}(\mathbb R^n)$ as the solution of the Neumann boundary problem
\begin{equation}\label{proj}-\Delta Pu+\mu Pu=-\Delta u \ \hbox{in}\  \Omega,\quad \partial_\nu u=0\ \hbox{on}\ \partial \Omega.\end{equation} Our main result reads as follows. 
\begin{theorem}\label{main}
Let $N \ge 4$, $p \ge q > 1$   and $q \ge \frac{4}{N-2}$. 
Let $\xi_0 \in \partial \Omega$ be a non degenerate critical point of the mean curvature $H$ with $H(\xi_0)<0$. 
Then if $\epsilon $ is small enough, the system \eqref{system} admits a solution $(u_\epsilon , v_\epsilon )$ which blows-up at $\xi_0$ as $\epsilon $ goes to zero, i.e. there exist $\xi_\epsilon \in\partial\Omega,$ $\xi_\epsilon \to\xi_0$ and $\delta_\epsilon \to0$ such that
$$\left\|(u_\epsilon ,v_\epsilon )- \left(PU_{\delta_\epsilon ,\xi_\epsilon },PV_{\delta_\epsilon ,\xi_\epsilon }\right)\right\|_X\to0
\ \hbox{as}\ \epsilon \to0.$$
\end{theorem}
 
 Let us make some comments concerning our result and the listed questions.
 
 \begin{remark}
 We point out that   if the domain is not convex   the mean curvature achieves its minimum at a point with negative curvature.
 We also observe that the assumption on the non-degeneracy of the critical points of the mean curvature is generically true.
  Indeed the mean curvature is a Morse function   for generic domains as proved in \cite{mipi}.
\\ Therefore, we can claim that for a generic not convex domain the slightly subcritical system \eqref{system} has a solution which blows-up at    the minimum point of the mean curvature of the boundary.
 \end{remark}
 
 \begin{remark}
 The proof relies on a classical Lyapunov--Schmidt reduction method. As usual, the starting point of the procedure consists in finding a good ansatz, which is
 the projection of the bubbles  \eqref{bubbles} solutions to the critical system \eqref{crisys} onto the Banach space $X$ defined in \eqref{xxx} in order to fit the Neumann boundary condition. A careful expansion of this projection is carried out in the Appendix \ref{app}. 
Next,   the key ingredient   turns out to be the {\em non-degeneracy} of the bubbles which has been proved in \cite{fkp}.  \\
It is worthwhile to point out  this argument cannot be used to produce    concentrating solutions in the sub-critical regime, because the non-degeneracy of the solution of the sub-critical system \eqref{subsys} (which is the natural ansatz)  is  a challenging open problem.
 \end{remark}
 
\begin{remark}
The requirement $p, q>1$ is necessary in order to perform the Lyapunov--Schmidt method. The   restriction on $q$, namely $q\ge2$ if $N=4$ or $q\ge\frac43$ if $N=5$, is due to technical reasons  and corresponds to requiring $\gamma \ge 1$ (see \eqref{gamma}) in the expansion in Proposition \ref{prop:stima phie}. Indeed, if $\gamma<1$ the first term in the expansion of $PU_{\delta_\epsilon ,\xi_\epsilon } -U_{\delta_\epsilon ,\xi_\epsilon }$ is no longer leading over the error term. 
\end{remark}

\begin{remark}
We observe that in  \cite{guopeng} the authors consider a slightly super-critical problem when $\mu=0$ and the domain $\Omega$ is a ball.
We strongly believe that combining their ideas with ours (see also Remark \ref{sopra}) it would be possible to prove that the system \eqref{system} in the slightly super-critical regime, i.e. $\epsilon $ is a {\em negative} small parameter,  has a solution which blows-up at a non-degenerate critical point of the mean curvature with positive level. This would extend the results
proved in \cite{dmp,rewe} for the almost critical equation to the almost critical system. In particular,   Remark \ref{sopra} suggests that the smallest exponent  $q$ (in our case we are assuming $p\ge q$) settles on  the sub-criticality or the super-criticality of the system.
\end{remark}

The proof relies on the Lyapunov-Schmidt procedure as developed in \cite{kimpi}.  To make reading simpler, we decide to sketch the main steps of the proof and  omit all the details which can be found up to minor modifications in  \cite{kimpi,r,rewe}. We only show what can not be found in the  literature.\\

\section{Proof of Theorem \ref{main}} 
\subsection{Setting of the problem}
We look for solutions to \eqref{system} in  the space $X$ defined in \eqref{xxx}. Given the embedding $\mathcal I:X\hookrightarrow L^{q+1\over q}(\Omega)\times L^{p+1\over p}(\Omega)$  we define its  formal adjoint   $\mathcal I^*: L^{q+1 }(\Omega)\times L^{p+1}(\Omega)\to X$ as
$$\mathcal I^*(f,g)=(u,v)\ \hbox{if and only if}\ \left\{\begin{aligned}&-\Delta u + \mu u=f \ \hbox{in} \ \Omega\\
&-\Delta v + \mu v=g \ \hbox{in} \ \Omega\\  &\partial_\nu u=\partial_\nu v=0 \ \hbox{on}\ \partial\Omega.\end{aligned}\right.$$
By classical regularity results  \cite[Lemma 2.1, Remark 2.1--2.2]{CK} and \cite[Theorem 15.2]{Agmon} we have that $\mathcal I^*$ is well defined and continuous, i.e.
$$\|\mathcal I^*(h,g)\|_X\le C \left(\|h\|_{q+1}+\|g\|_{p+1}\right).$$
where the constant $C$ only depends on $\mu,p,q$ and $\Omega.$
\\
Let us rewrite problem \eqref{system} as
\begin{equation}\label{sys-re}
(u,v)=\mathcal I^*(f_\epsilon(v),g_\epsilon(u)),\ \hbox{with}\ f_\epsilon(v):=(v^+)^{q-\epsilon}\ \hbox{and}\ g_\epsilon(u):=(u^+)^{p-\epsilon}
\end{equation}
The positivity of the solution   follows by the maximum principle. Finally, we look for a solution to \eqref{sys-re} as
\begin{equation}\label{ans}(u,v)=\left(PU_{\delta, \xi},P V_{\delta, \xi}\right)+(\phi,\psi),\ \hbox{with}\ 
 \delta=d\epsilon\ \hbox{being}\ d>0\ \hbox{and}\ \xi\in\partial\Omega.\end{equation}
Here $\left(PU_{\delta, \xi},P V_{\delta, \xi}\right):=\mathcal I^*\left(f_0(V_{\delta, \xi}),g_0(U_{\delta, \xi})\right)$ (see also \eqref{proj})
and $(U_{\delta, \xi}, V_{\delta, \xi})$ is the bubble defined in \eqref{bubbles}.
The choice of $\delta$ is motivated by the expansion of the reduced energy in \eqref{expansion energy}.
The higher order term $(\phi,\psi)$ belongs to the  space $Z_{\delta, \xi}$ defined as follows.
\\ First, it is useful to remind that all the solutions to the linear system
\[ \begin{cases}
-\Delta \Psi =p \ud^{p-1} \Phi &\text{ in } \R^N \\
-\Delta \Phi =q \vd^{q-1} \Psi &\text{ in } \R^N.
\end{cases} \]
  are linear combination of (see \cite{fkp})
\[ \left( \frac{\partial \ud}{\partial \delta},  \frac{\partial \vd}{\partial \delta}\right)\ \hbox{and} \ \left( \frac{\partial \ud}{\partial \xi_i},  \frac{\partial \vd}{\partial \xi_i}\right),\  i=1,\dots,N.\]
Now, since the point $\xi$ belongs to the boundary we shall choose an orthonormal system of coordinates $\tau_1,\dots,\tau_{N-1}$ of the tangent space to $\partial\Omega$ at $\xi$ and  define the functions
$$(\Phi_{\delta, \xi}^0,\Psi_{\delta, \xi}^0):=\left( \frac{\partial \ud}{\partial \delta},  \frac{\partial \vd}{\partial \delta}\right)\ \hbox{and}\ 
(\Phi_{\delta, \xi}^i,\Psi_{\delta, \xi}^i):=\left( \frac{\partial \ud}{\partial \tau_i},  \frac{\partial \vd}{\partial \tau_i}\right),\  i=1,\dots,N-1.$$
Set $\left(P\Phi_{\delta, \xi}^i,P \Psi_{\delta, \xi}^i\right):=\mathcal I^*\left(p \ud^{p-1} \Phi_{\delta, \xi}^i  , q \vd^{q-1} \Psi_{\delta, \xi}^i\right)$
and define the space
\[ Y_{\delta, \xi}:= \text{span} \left\{ \left(P \Phi_{\delta, \xi}^i, P \Psi_{\delta, \xi}^i\right) ,\ i=0,1,\dots,N-1\right\} \]
 and  
\begin{align*} Z_{\delta, \xi}&=\{ (\phi, \psi) \in X: \,p \int_\Omega \ud^{p-1} \Phi_{\delta, \xi}^i \phi +  q\int_\Omega \vd^{q-1} \Psi_{\delta, \xi}^i \psi=0,  \forall \ i \} \\
& =\{ (\phi, \psi) \in X: \int_\Omega \nabla P \Psi_{\delta, \xi}^i  \nabla \phi +\mu \int_\Omega P \Psi_{\delta, \xi}^i \phi + \int_\Omega \nabla P \Phi_{\delta, \xi}^i  \nabla \psi +\mu \int_\Omega P \Phi_{\delta, \xi}^i \psi=0, \  \forall\ i  \}. \end{align*}
Next, we introduce the projections $\Pi_{\delta,\xi}:X\to Y_{\delta, \xi}$ and $\Pi^\perp_{\delta,\xi}:X\to Z_{\delta, \xi}$
and write the single equation \eqref{sys-re} as a system 
\begin{equation}\label{sys-rere-1}
\Pi^\perp_{\delta,\xi}\left\{\left(PU_{\delta, \xi}+\phi,P V_{\delta, \xi}+\psi\right)-\mathcal I^*\left[\left(f_\epsilon\left(P V_{\delta, \xi}+\psi\right),g_\epsilon\left(PU_{\delta, \xi}+\phi\right)\right)\right]\right\}=0
\end{equation}
and
\begin{equation}\label{sys-rere-2}
\Pi _{\delta,\xi}\left\{\left(PU_{\delta, \xi}+\phi,P V_{\delta, \xi}+\psi\right)-\mathcal I^*\left[\left(f_\epsilon\left(P V_{\delta, \xi}+\psi\right),g_\epsilon\left(PU_{\delta, \xi}+\phi\right)\right)\right]\right\}=0.
\end{equation}
\subsection{Solving (\ref{sys-rere-1})}
First of all, we write \eqref{sys-rere-1}
as 
$$\mathscr L (\phi,\psi)=\mathscr N(\phi,\psi)+\mathscr R$$
where the linear operator $\mathscr L $ is defined as
$$\mathscr L (\phi,\psi):=\Pi^\perp_{\delta,\xi}\left\{\left( \phi, \psi\right)-\mathcal I^*\left[\left(f'_\epsilon\left(P V_{\delta, \xi}\right)\psi,g'_\epsilon\left(PU_{\delta, \xi}\right)\phi\right)\right]\right\},
$$
the nonlinear term $\mathscr N$ is defined as
$$\begin{aligned}\mathscr N(\phi,\psi):=&\Pi^\perp_{\delta,\xi}\left\{
\mathcal I^*\left[
\left(f_\epsilon\left(P V_{\delta, \xi}+\psi\right)-f_\epsilon\left(P V_{\delta, \xi}\right)-f'_\epsilon\left(P V_{\delta, \xi}\right)\psi,\right.\right.\right.
\\ &\left.\left.\left.\qquad\qquad g_\epsilon\left(PU_{\delta, \xi}+\phi\right)-g_\epsilon\left(PU_{\delta, \xi} \right)-g'_\epsilon\left(PU_{\delta, \xi}\right)\phi \right)
\right]\right\}\end{aligned}$$
and the remainder term $\mathscr R$ is 
$$\mathscr R:=\Pi^\perp_{\delta,\xi}\left\{\mathcal I^*\left[\left(f_\epsilon\left(P V_{\delta, \xi}\right)-f_0\left( V_{\delta, \xi}\right),
g_\epsilon\left(P U_{\delta, \xi}\right)-g_0\left( U_{\delta, \xi}\right)
\right)\right]\right\}.$$
Using quite standard arguments we can prove the invertibility of the linear operator $\mathscr L$.
\begin{proposition}
For any $\eta>0$ there exists $\epsilon_0>0$ and   $C>0$ such that for any $\epsilon\in(0,\epsilon_0)$,  $d\in[\eta,1/\eta]$ and $\xi\in\partial\Omega$
$$\|\mathscr L(\phi,\psi)\|_X\ge C\|(\phi,\psi)\|_X\ \hbox{for any }\ (\phi,\psi)\in Z_{\delta,\xi}.$$
\end{proposition}
The key estimate is the size of the remainder term 
\begin{proposition}\label{stima-errore}
For any $\eta>0$ there exists $\epsilon_0>0,$  $\sigma>0$ and   $C>0$ such that for any $\epsilon\in(0,\epsilon_0)$,  $d\in[\eta,1/\eta]$ and $\xi\in\partial\Omega$
$$\|\mathscr R\|_X\le C\epsilon^{\frac12+\sigma}.$$
\end{proposition}
\begin{proof}
We observe that   
\[ \norm{\mathscr R}_X \le C \norm{(P\ud)^{p-\epsilon } - \ud^p}_{\frac{p+1}{p}} +\norm{(P \vd)^{q-\epsilon } - \vd^q}_{\frac{q+1}{q}}. \]
Let us estimate  the first term. One has
\begin{equation}\label{stima errore 1} \norm{(P\ud)^{p-\epsilon } - \ud^p}_{\frac{p+1}{p}}  \le \norm{(P\ud)^{p-\epsilon } - (P\ud)^{p}}_{\frac{p+1}{p}} + \norm{(P\ud)^{p} - \ud^p}_{\frac{p+1}{p}}.  \end{equation}
Notice that for some $\theta\in(0,1) $ 
\begin{multline}\label{stima log} \int_\Omega \left( (P\ud)^{p-\epsilon } - (P\ud)^p \right)^{\frac{p+1}{p}} \le \epsilon ^{\frac{p+1}{p}} \int_\Omega |P\ud|^{p+1} (\ln P\ud)^{\frac{p+1}{p}} \\+ \epsilon ^{2\frac{p+1}{p}} \int_\Omega |P\ud|^{(p-\theta\epsilon)\frac{p+1}{p}} (\ln P\ud)^{2 \frac{p+1}{p}}.  \end{multline}
Moreover, for any $t > p+1$, by Corollary \ref{cor:stima phie} and \eqref{impo1}, 
\begin{align*} \int_\Omega |P\ud|^{p+1} (\ln P\ud)^{\frac{p+1}{p}} &\le C\, \int_\Omega |\ud|^{p+1} (\ln \ud)^{\frac{p+1}{p}}  \\
& \le C\, \int_\Omega \ud^{t }= 
\begin{cases}
C \delta^{-\frac{Nt}{p+1} + N} & \text{ if } N < (\gamma+1)t \\
C \delta^{-\frac{Nt}{p+1} + N} \ln \delta & \text{ if } N = (\gamma+1)t\\
C \delta^{-\frac{Nt}{p+1} + (\gamma+1)t} & \text{ if } N > (\gamma+1)t. 
\end{cases}
  \end{align*}
where $\gamma$ is defined in \eqref{gamma}. 
 
Similarly, we can estimate the second term in \eqref{stima log}.
As a consequence, 
\[ \int_\Omega \left( (P\ud)^{p-\epsilon } - (P\ud)^p \right)^{\frac{p+1}{p}} \le \begin{cases}
C\,  \delta^{\frac{p+1}{p}} \delta^{-\frac{Nt}{p+1} + N} & \text{ if } N < (\gamma+1)t \\
C\,  \delta^{\frac{p+1}{p}} \delta^{-\frac{Nt}{p+1} + N} \ln \delta & \text{ if } N = (\gamma+1)t\\
C\,  \delta^{\frac{p+1}{p}} \delta^{-\frac{Nt}{p+1} + (\gamma+1)t} & \text{ if } N > (\gamma+1)t. 
\end{cases}
\]
We finally conclude that for any $\beta \in (0, \frac 12)$, we can take $t >p+1$ small enough such that 
\begin{equation}\label{stima errore 2} \norm{(P\ud)^{p-\epsilon } - (P\ud)^{p}}_{\frac{p+1}{p}} \le C\, \delta^{1- \frac{Nt p}{(p+1)^2} + \frac{Np}{p+1}} \le C \delta^{\frac 12 + \beta}. \end{equation}

To estimate the second term in \eqref{stima errore 1}, we use the expansion in Corollary \ref{cor:stima phie}, to get
\[  \norm{(P\ud)^{p} - \ud^{p}}_{\frac{p+1}{p}} \le \norm{\ud^{p-1} \delta^{-\frac{N}{p+1} +1} \varphi_0(x/\delta)}_{\frac{p+1}{p}} + \norm{ \ud^{p-1} \zeta_\delta}_{\frac{p+1}{p}}.
\]
Now, 
\[ \delta^{-\frac{N}{p} +\frac{p+1}{p}} \int_\Omega \left( \ud^{p-1} \varphi_0((x-\xi)/\delta) \right)^{\frac{p+1}{p}} \le 
\begin{cases}
C\,  \delta^{\frac{p+1}{p}-N+(\gamma+1)(p-1)\frac{p+1}{p} + \gamma\frac{p+1}{p}} & \text{ if } p \gamma + p -1 <\frac{Np}{p+1} \\
C\, \delta^{\frac{p+1}{p}} \ln(\delta) & \text{ if } p \gamma + p -1  = \frac{Np}{p+1} \\
C\, \delta^{\frac{p+1}{p}} & \text{ if } p \gamma + p -1  >\frac{Np}{p+1}
\end{cases} 
\]
thus
\[ \norm{\ud^{p-1} \delta^{-\frac{N}{p+1} +1} \varphi_0((x-\xi)/\delta)}_{\frac{p+1}{p}}  
\le 
\begin{cases}
C\, \delta^{\left(\gamma+1 -\frac{N}{p+1}\right) p} & \text{ if } p \gamma + p -1  <\frac{Np}{p+1} \\
C\, \delta (\ln(\delta))^\frac{p}{p+1} & \text{ if } p \gamma + p -1  = \frac{Np}{p+1} \\
C\, \delta & \text{ if }p \gamma + p -1  >\frac{Np}{p+1}. 
\end{cases} 
\]
Similarly,
\[ \norm{ \ud^{p-1} \zeta_\delta}_{\frac{p+1}{p}} \le 
\begin{cases}
C\, \delta^{\left(\gamma+1 -\frac{N}{p+1}\right) p}& \text{ if } p \gamma + p -1  <\frac{Np}{p+1} \\ 
C\, \delta (\ln(\delta))^\frac{p}{p+1} & \text{ if } p \gamma + p -1  = \frac{Np}{p+1} \\
C\, \delta & \text{ if }p \gamma + p -1  >\frac{Np}{p+1}. 
\end{cases} \]
Thus 
one has 
\[ \norm{(P\ud)^{p} - \ud^{p}}_{\frac{p+1}{p}}\le 
\begin{cases}
C\, \delta^{\left(\gamma+1 -\frac{N}{p+1}\right) p}& \text{ if } \left(\gamma+1 -\frac{N}{p+1}\right) p < 1  \\ 
C\, \delta (\ln(\delta))^\frac{p}{p+1} & \text{ if } \left(\gamma+1 -\frac{N}{p+1}\right) p = 1 \\
C\, \delta & \text{ if } \left(\gamma+1 -\frac{N}{p+1}\right) p > 1. 
\end{cases}
 \] 
This, combined with \eqref{stima errore 2}, and since the terms with $\vd$ can be treated similarly, gives
\[ \norm{\mathscr R}_X \le C\,( \delta^{\min \left \{ \left(\gamma+1 -\frac{N}{p+1}\right) p, 1\right\}} (\ln(\delta))^{s_1} + \delta^{\min \left \{ \frac{Nq}{p+1}, 1\right\}} (\ln(\delta))^{s_2} + \delta^{\frac 12 + \beta}), \]
where $s_1=\frac{p}{p+1}$ if $\left(\gamma+1 -\frac{N}{p+1}\right) p =1$, $s_1=0$ otherwise; whereas $s_2=\frac{q}{q+1}$ if $\frac{Nq}{p+1}=1$, and $s_2=0$ otherwise.
The conclusion follows observing that
$$ \frac{pqN}{p+1} > \frac{qN}{p+1} > \frac 12, \qquad \frac{pN}{q+1} >\frac 12, $$
which are verified if
\begin{equation}\label{condizione e} p \ge q > \frac{5+ \sqrt{8N+9}}{4(N-2)}. \end{equation}
Notice that this quantity is $\le 1$ if $N \ge 5$, and $< \frac{4}{N-2}$ if $N=4$, thus \eqref{condizione e} is satisfied under the hypotheses of Theorem \ref{main}. 
\end{proof}
Finally, using a classical contraction mapping argument we  solve \eqref{sys-rere-1}.
\begin{proposition}\label{stima E}
For any $\eta>0$ there exists $\epsilon_0>0,$  $\sigma>0$ and   $C>0$ such that for any $\epsilon\in(0,\epsilon_0)$,  $d\in[\eta,1/\eta]$ and $\xi\in\partial\Omega$ there exists a unique $(\phi,\psi)=(\phi(d,\xi,\epsilon),\psi(d,\xi,\epsilon))\in Z_{\delta,\xi}$ which solves \eqref{sys-rere-1} and
satisfies
$$\|(\phi,\psi)\|_X\le C\epsilon^{\frac12+\sigma}.$$
Moreover, $(d,\xi)\to (\phi(d,\xi,\epsilon),\psi(d,\xi,\epsilon))$ is a $C^1-$map.
\end{proposition}
\subsection{Solving (\ref{sys-rere-2})}
We introduce the functional 
\[ J_\epsilon(u, v):=\frac12 \int_\Omega \nabla u \nabla v +\frac \mu 2 \int_\Omega uv - \frac{1}{p+1-\varepsilon} \int_\Omega (u^+)^{p+1-\varepsilon} - \frac{1}{q+1-\varepsilon} \int_\Omega (v^+)^{q+1-\varepsilon}, \] 
whose critical points are solutions to the system \eqref{system}.
We also define the reduced functional
\[ \tilde J_\epsilon : (0,+\infty)\times\partial\Omega\to \mathbb R, \quad \tilde J_\epsilon(d, \xi):=J_\epsilon(P \ud+\phi, P \vd+\psi) \]
where $\delta=d\xi$ as in \eqref{ans} and $(\phi, \psi)\in Z_{\delta,\xi}$ has been already found in   Proposition \ref{stima E}.
As it is usual, solving \eqref{sys-rere-2} turns out to be equivalent to finding a critical point of the reduced energy, as stated in  following result.
\begin{proposition}\label{reduced funct}
There exists $\epsilon_0>0$    such that for any $\epsilon\in(0,\epsilon_0)$
if $(d, \xi)$ is a critical point for $\tilde J_\epsilon$ then $(P \ud+\phi, P \vd+\psi)$ is a critical point of $J_\epsilon$ and in particular a solution to \eqref{sys-rere-2}.
\\  Moreover
	\begin{equation}\label{expansion energy} \tilde J_\epsilon(d, \xi)= \mathfrak c_1 - \mathfrak c_2 \eps \ln \eps +\mathfrak c_3\epsilon+\underbrace{\left(-\mathfrak c_4 H(\xi) d  -\mathfrak c_2 \ln d  \right)}_{=\Theta(d,\xi)}\epsilon + o(\eps), \end{equation}
 $C^1-$uniformly with respect to $\xi\in\partial\Omega$ and $d$ in compact sets of $(0,\infty)$.
 Here the $\mathfrak c_i$'s are constants and in particular $\mathfrak c_2$ and  $\mathfrak c_4$ are positive.
 \end{proposition}
\begin{proof}
It is standard to prove the first part of the claim and also to show that by Proposition \ref{stima E}
$$ J_\epsilon(P \ud+\phi, P \vd+\psi)=J_\epsilon(P \ud , P \vd )+o(\epsilon),$$
so we need only to estimate the term  $J_\epsilon(P \ud , P \vd )$.
\\ We have for any $\lambda \in [0, 1]$
$$\begin{aligned} J_\epsilon(PU_{\delta, \xi} , PV_{\delta,\xi})&= \int_\Omega \nabla PU_{\delta, \xi} \nabla PV_{\delta, \xi} + \mu \int_\Omega PU_{\delta, \xi}PV_{\delta, \xi} \\
&\qquad  - \frac{1}{p+1- \eps} \int_\Omega (PU_{\delta, \xi})^{p+1- \eps} - \frac{1}{q+1-\eps} \int_\Omega (PV_{\delta, \xi})^{q+1- \eps} \\
&=  \lambda \int_\Omega U_{\delta, \xi}^p PU_{\delta, \xi} + (1-\lambda) \int_\Omega V_{\delta, \xi}^q PV_{\delta, \xi} \\
&\qquad  - \frac{1}{p+1- \eps} \int_\Omega (PU_{\delta, \xi})^{p+1- \eps} - \frac{1}{q+1- \eps} \int_\Omega (PV_{\delta, \xi})^{q+1- \eps}.
\end{aligned}$$
Let us estimate separately each one of the previous four terms. \\
We can write (see  Corollary \ref{cor:stima phie})
\[ \int_\Omega U_{\delta, \xi}^p PU_{\delta, \xi} = \int_\Omega U_{\delta, \xi}^{p +1} + \delta^{-\frac{N}{p+1} +1}\int_\Omega U_{\delta, \xi}^p\varphi_0\left( \frac{x-\xi}{\delta} \right) + \int_\Omega U_{\delta, \xi}^p \zeta_\delta. \]
First of all by \cite[Lemma 4.7]{PST}
$$\int_\Omega U_{\delta, \xi}^{p+1} = \frac12 S_{p,q}^{\frac N2} -C_1 H(\xi) \delta  + o(\delta), $$
with 
\[ C_1:= \frac 12 \int_{\R^{N-1}} |y'|^2 U^{p+1}(y', 0) \, dy'>0. \] 
Moreover, for any fixed $\rho >0$, 
\begin{align*} \delta^{-\frac{N}{p+1} +1}\int_{\Omega \cap B_\rho(0)} U_{\delta, \xi}^p \varphi_0\left( \frac{x-\xi}{\delta} \right) &= \delta^{-\frac{N}{p+1} +1} \int_{\Omega \cap B_\rho(0)} \delta^{-\frac{Np}{p+1}} U^p\left( \frac{x-\xi}{\delta} \right)\varphi_0\left( \frac{x-\xi}{\delta} \right)\\
&=\delta \int_{\R^N_+} U^p(x) \varphi_0(x) + o(\delta) \\
&= \delta \int_{\partial \R^N_+} (\partial_\nu \varphi_0) V + o(\delta) \\
&=  - \delta \int_{\partial \R^N_+} U' V \frac{\sum \rho_j x_j^2}{|x|} + o(\delta)\\ &=C_3 H(\xi) \delta + o(\delta), 
 \end{align*}
 where (see  \cite[Lemma B.8]{PSaT})
 \[ C_3=-\frac{N-1}{2} \int_{\partial \R^N_+} U' V \frac{x_1^2}{|x|} = -\frac 12 \int_{\R^{N-1}} |y'| U' V >0, \]  
and 
\[  \delta^{-\frac{N}{p+1} +1}\int_{\Omega \cap B^c_\rho(0)} U_{\delta, \xi}^p \varphi_0\left( \frac{x-\xi}{\delta} \right)  \le
 C \delta^{-\frac{N}{p+1} +1} \delta^{-\frac{Np}{p+1} + \gamma p+p + \gamma } = C\, \delta^{-N +(\gamma+1)(p+1)} =o(\delta).  \]
Furthermore, 
$$ \int_\Omega U_{\delta, \xi}^p \zeta_\delta \le C \delta^{-\frac{N}{p+1} + 2} (\ln \delta)^{\hat \sigma} \int_\Omega U_{\delta, \xi}^p\le \delta^2 (\ln \delta)^{\hat \sigma}  \int_{\Omega_\delta} U_{1, 0}^p(x) =o(\delta). 
$$
Collecting all the previous estimates
\[ \int_\Omega U_{\delta, \xi}^p PU_{\delta, \xi} = \frac12 S_{p,q}^{\frac N2} -C_1 H(\xi) \delta + C_3 H(\xi) \delta  + o(\delta). \]
and similarly
\[ \int_\Omega V_{\delta, \xi}^q PV_{\delta, \xi} = \frac12 S_{p,q}^{\frac N2} -C_2 H(\xi) \delta + C_4 H(\xi) \delta  + o(\delta), \]
where
\[ C_2:= \frac 12 \int_{\R^{N-1}} |y'|^2 V^{q+1}(y', 0) \, dy'>0\ \hbox{and} \ C_4:=-\frac 12 \int_{\R^{N-1}} |y'| V' U >0. \]
Next we estimate the terms coming from the nonlinearity. 
One has
\[ \frac{1}{p+1- \eps}= \frac{1}{p+1}+  \frac{\eps}{(p+1)^2} + o(\eps) \]
and also
\begin{align*} \int_\Omega (PU_{\delta, \xi})^{p+1- \eps}&= \int_\Omega (PU_{\delta, \xi})^{p+1} + ((PU_{\delta, \xi})^{p+1-\eps}-(PU_{\delta, \xi})^{p+1})  \\
&=  \int_\Omega (PU_{\delta, \xi})^{p+1} - \eps  \int_\Omega (PU_{\delta, \xi})^{p+1} \ln(PU_{\delta, \xi}) +  \eps^2 \int_\Omega (PU_{\delta, \xi})^{p+1} (\ln(PU_{\delta, \xi}))^2 + o(\eps^3). \end{align*}
Arguing as above, the first term can be estimated as 
 \[ \int_\Omega (PU_{\delta, \xi})^{p+1} = \frac 12 S_{p, q}^{\frac{N}{2}} - C_1 H(\xi) \delta + (p+1)C_3 H(\xi) \delta + o(\delta). \]
while the second one  can be estimated by   \cite[(B.1)  and Lemma 4.10]{PSaT} as
\begin{align*}
 \int_\Omega (PU_{\delta, \xi})^{p+1} \ln(PU_{\delta, \xi}) &=   \int_\Omega U_{\delta, \xi}^{p+1} \ln(U_{\delta, \xi}) +o(1) \\
&= \int_{\Omega_\delta} U^{p+1} \ln (U) - \frac{N}{p+1} \int_{\Omega_\delta} U^{p+1} \ln(\delta)  + o(1)\\
&=  C_5 - \frac{N}{p+1} \left( \frac 12 S_{p,q}^{\frac N2} \ln(\delta) - C_1 H(\xi)\delta  \ln(\delta) \right)+ o(1),
\end{align*}
with 
\[ C_5:= \frac 12 \int_{\R^N} U^{p+1} \ln (U). \]
Therefore, again collecting the previous estimates
 \begin{align*} \frac{1}{p+1- \eps} \int_\Omega (PU_{\delta, \xi})^{p+1-\eps}& = \frac{1}{2(p+1)} S_{p,q}^{N/2} - \frac{H(\xi)}{p+1} \left(C_1 - (p+1)C_3 \right)\delta  \\
 & \qquad -\frac{ \eps}{p+1} \left(C_5-\frac{1}{2(p+1)}S_{p,q}^{N/2}\right) +  \frac{ N}{2(p+1)^2} S_{p,q}^{N/2} \eps \ln(\delta) \\
 & \qquad+ o(\eps) + o(\delta)+ O(\delta \eps) 
 \end{align*}
and similarly
 \begin{align*}
 \frac{1}{q+1- \eps} \int_\Omega (PV_{\delta, \xi})^{q+1- \eps} & =  \frac{1}{2(q+1)} S_{p,q}^{N/2} - \frac{H(\xi)}{q+1} \left(C_2 - (q+1)C_4 \right)\delta \\
 & \hspace{0.5cm} - \frac{\eps}{q+1} \left(C_6-\frac{1}{2(q+1)}S_{p,q}^{N/2}\right) + \frac{ N}{2(q+1)^2} S_{p,q}^{N/2} \eps \ln(\delta) \\
 & \hspace{3cm} + o(\eps) + o(\delta)+ O(\delta \eps).
 \end{align*}
Finally we get
\begin{equation}\begin{aligned}
J(PU_{\delta, \xi} , PV_{\delta,\xi})& =\underbrace{ \frac{S_{p,q}^{N/2}}{N} }_{=:\mathfrak c_1} - H(\xi)\delta
\underbrace{\left[C_1 \left(\lambda- \frac{1}{p+1}\right) + C_2   \left(1-\lambda -\frac{1}{q+1}\right) +(1-\lambda)  C_3+ \lambda  C_4\right]}_{=:\mathfrak c_4}\\
& + \eps\underbrace{\left[\left(\frac{ C_5}{p+1} + \frac{ C_6}{q+1}\right) -\frac{S_{p,q}^{N/2}}2 \left(\frac{1}{(p+1)^2} 
 + \frac{1}{(q+1)^2}\right)\right] }_{=:\mathfrak c_3}\\
 &   - \eps \ln(\delta)\underbrace{\frac{N}{2} S_{p,q}^{N/2} \left(\frac{1 }{(p+1)^2} + \frac{1 }{(q+1)^2}\right) }_{=:\mathfrak c_2} + o(\eps) + o(\delta) + O(\delta\eps)\label{energia}
\end{aligned} \end{equation}
and choosing  $\delta=d\eps$ the claim follows. 
\\

We remark that the constant $\mathfrak c_4$ (which a priori depends on  $\lambda$) is strictly positive  whenever 
\[  \frac{1}{p+1} < \lambda < \frac{q}{q+1}. \]
Actually we show that it is independent of $\lambda$, because its derivative with respect to $\lambda$ is zero, i.e.
\[ +C_1 - C_2 - C_3 + C_4 = 0. \]
Indeed by a straightforward computation  
\begin{align*}
\int_{\R^{N-1}} |y'|^2 U^{p+1}(y', 0) \, dy' &= \int_{\R^{N-1}} |y'|^2 U(y', 0) (-\Delta V)(y', 0) \, dy' \\
&=\frac{1}{\omega_{N-2}} \int_0^{+\infty} r^N U(r) (-V''(r) - \frac{N-1}{r} V'(r)) \, dr \\
&= - \frac{1}{\omega_{N-2}}  \int_0^{+\infty} r^N U V'' - \frac{N-1}{\omega_{N-2}}   \int_0^{+\infty} r^{N-1} U V'
\end{align*}
and by integration by parts, and using \eqref{impo1} and \eqref{impo2} 
\[  \int_0^{+\infty} r^N U V''= \int_0^{+\infty} (r^N U)'' V = N(N-1)  \int_0^{+\infty} r^{N-2} UV + 2 N  \int_0^{+\infty} r^{N-1} U' V +  \int_0^{+\infty} r^N U''V . \]
Moreover,
\[ C_3=-\frac 12 \int_{\R^{N-1}} |y'| U' V= -\frac{1}{2 \omega_{N-2}} \int_0^{+\infty} r^{N-1} U' V \, dr, \]
and a similar expression holds for $C_4$. 
Thus
\begin{align*}
&-C_1 + C_2 + C_3 - C_4\\  &= \frac 1{2 \omega_{N-2}} \left(N(N-1)  \int_0^{+\infty} r^{N-2} UV + 2 N  \int_0^{+\infty} r^{N-1} U' V +  \int_0^{+\infty} r^N U''V  \right) \\
&   + \frac{N-1}{2 \omega_{N-2}}   \int_0^{+\infty} r^{N-1} U V' - \frac{1}{2 \omega_{N-2}}  \int_0^{+\infty} r^N U'' V - \frac{N-1}{2\omega_{N-2}}   \int_0^{+\infty} r^{N-1} U' V  \\
&  -\frac{1}{2 \omega_{N-2}} \int_0^{+\infty} r^{N-1} U' V +\frac{1}{2 \omega_{N-2}} \int_0^{+\infty} r^{N-1} U V' \\
&=  \frac 1{2 \omega_{N-2}} \left( N(N-1)  \int_0^{+\infty} r^{N-2} UV + N   \int_0^{+\infty} r^{N-1} U' V  + N  \int_0^{+\infty} r^{N-1} U V'  \right)\\ &=0,
\end{align*}
because
\[ \int_0^{+\infty} r^{N-1} U' V = - \int_0^{+\infty} (r^{N-1} V)'U = - (N-1) \int_0^{+\infty} r^{N-2} V U - \int_0^{+\infty} r^{N-1} V'U. \qedhere \]
\end{proof}

Finally, we can complete the proof of Theorem \ref{main}.
\begin{proof} Let $\xi_0$ be a non-degenerate critical point of $H$ with $H(\xi_0)<0$. Set $d_0:=-\frac{\mathfrak c_5}{\mathfrak c_4 H(\xi_0)}.$
Therefore $(d_0,\xi_0)$ is a non-degenerate critical point of $\Theta$ which is stable under $C^1-$perturbation. Thus the reduced energy has a critical point  $(d_\epsilon,\xi_\epsilon)\to  (d_0,\xi_0) $  as $\epsilon\to0.$
Collecting all the previous results, the claim follows.
\end{proof}

\begin{remark}\label{sopra}
The expansion of the reduced energy  suggests that we could find a positive solution to 
 the almost critical problem
$$ 
-\Delta u + \mu u=v^{q\pm \epsilon },\
-\Delta v+ \mu v=u^{p\pm \epsilon }\  \text{ in $\Omega$},\ 
\partial_\nu u= \partial_\nu v=0   \text{ on $\partial \Omega$},
$$
which blows-up at a critical point $\xi_0$ of the mean curvature if
either {\em  $H(\xi_0)<0$ and the exponents are $(q-\epsilon,p\pm\epsilon)$} or
{\em $H(\xi_0)>0$ and the exponents are $(q+\epsilon,p\pm\epsilon)$
}.\\
It is enough to observe that the sign of the
  $\epsilon\ln\delta$ term in \eqref{energia} is determined by the  constant $ \mathfrak c_2$ 
which (up to a positive constant) is
$$\mathfrak c_2=\left\{\begin{aligned} & \frac{1 }{(p+1)^2} + \frac{1 }{(q+1)^2} >0 \ \hbox{in the case }\ (q-\epsilon,p-\epsilon)\\
& -\frac{1 }{(p+1)^2} - \frac{1 }{(q+1)^2} <0 \ \hbox{in the case }\ (q+\epsilon,p+\epsilon)\\ 
& \frac{1 }{(p+1)^2} - \frac{1 }{(q+1)^2} <0 \ \hbox{in the case }\ (q+\epsilon,p-\epsilon)\\
& -\frac{1 }{(p+1)^2} + \frac{1 }{(q+1)^2} >0 \ \hbox{in the case }\ (q-\epsilon,p+\epsilon).\\
\end{aligned}\right.$$
\end{remark}

\appendix
\section{Expansion of the projections}\label{app}
In this Section, we prove some crucial estimates on the projections $P\ud, P\vd$, and $P\phid, P\psid$, see also \cite{guopeng} for related results in case $\mu=0$.

First of all, it is useful to remind the decay of the bubble $(U,V)$ and of its derivative:
\begin{equation}\label{impo1}
  \lim _{r \rightarrow \infty} r^{N-2} V(r)=a\ \hbox{and}\  \begin{cases}
\lim _{r \rightarrow \infty} r^{N-2} U(r)=b  & \text { if } q>\frac{N}{N-2} \\ 
\lim _{r \rightarrow \infty} \frac{r^{N-2}}{\ln r} U(r)=b  & \text { if } q=\frac{N}{N-2} \\ 
\lim _{r \rightarrow \infty} r^{q(N-2)-2} U(r)=b  & \text { if } q<\frac{N}{N-2}
\end{cases}
\end{equation}
for some positive $a$ and $b$ and also
\begin{equation}\label{impo2}
 \lim _{r \rightarrow \infty} {rV'(r)\over V(r)}=\lim _{r \rightarrow \infty}  {rU'(r)\over U(r)}=1.
\end{equation}

Next, it is useful to introduce the constant
\begin{equation}\label{gamma}
\gamma:= 
\left\{\begin{aligned} &N-3\ &\hbox{if }\ q>\frac N{N-2}\\
 & q(N-2)-3\ &\hbox{if }\ q<\frac N{N-2}\\
\end{aligned}\right.
\end{equation}
and to point out that it satisfies
\begin{equation}\label{gamma2}
0<\gamma+1-\frac N{p+1}= \left\{\begin{aligned} &\frac N{q+1}\ &\hbox{if }\ q>\frac N{N-2}\\
 & \frac {Nq}{p+1}\ &\hbox{if }\ q<\frac N{N-2}.\\
\end{aligned}\right.
\end{equation}
Even if in the  paper we assume $\gamma \ge 1$, we prefer to give the proof for the more general case $\gamma >0$, on the one hand, to make more transparent the reason why the assumption $\gamma \ge 1$ is actually needed, and on the other hand, because we believe it could be helpful to gain a better understanding of the open case $0<\gamma <1$. 
\\

Let $\rho$ be the function which locally describes the boundary, namely for a suitably small $s$ we have
\[ \Omega \cap B_s(0)=\{ (x', x_N): x_N > \rho(x')\}, \qquad \rho(x')=\sum_{j=1}^{N-1} \rho_j x_j^2 + O(|x'|^3). \]
Notice that $H(0)=\frac{2}{N-1}\sum_{j=1}^{N-1} \rho_j$. 
Define 
\[ \begin{cases}
\Delta \varphi_0=0 & \R^N_+ \\
\frac{\partial \varphi_0}{\partial x_N} (x) = U'(|x|) \frac{\sum_{j=1}^{N-1} \rho_j x_j^2}{|x|}  & \partial \R^N_+\\
\varphi_0(x) \to 0 & |x| \to \infty 
\end{cases}\]
and
\[ \begin{cases}
\Delta \psi_0=0 & \R^N_+ \\
\frac{\partial \psi_0}{\partial x_N} (x) = V'(|x|)\frac{\sum_{j=1}^{N-1} \rho_j x_j^2}{|x|} & \partial \R^N_+\\
\psi_0(x) \to 0 & |x| \to \infty.
\end{cases}\]
\begin{lemma}\label{lem: stima phi}
One has
\[|\varphi_0(x)| \le \frac{C}{(1+|x|)^\gamma},  \quad \abs{\nabla \varphi_0(x)} \le \frac{C}{(1+|x|)^{\gamma+1}}, \quad \abs{D^2 \varphi_0(x)} \le \frac{C}{(1+|x|)^{\gamma+2}} \]
whereas
\[  |\psi_0(x)| \le \frac{C}{(1+|x|)^{N-3}},  \quad \abs{\nabla \psi_0(x)} \le \frac{C}{(1+|x|)^{N-2}}, \quad \abs{D^2 \psi_0(x)} \le \frac{C}{(1+|x|)^{N-1}}. \]
\end{lemma}
\begin{proof}
By using the Poisson kernel for the half-space, and $N \ge 4$, we get
\[ \varphi_0(x) = - C \int_{\R^{N-1}}  \frac{\sum_{j=1}^{N-1} \rho_j y_j^2 \, U'(|y|)}{|y| \, |x-y|^{N-2}  } \, dy \]
where $y=(y_1, \dots, y_{N-1}, 0)$, see for instance \cite{copa}. 
By \eqref{impo2}
\[ |U'(r) r | \le \frac{C}{(1+r)^{\gamma+1}} \quad \text{if } \quad q \ne \frac{N}{N-2},\]
so that by
\[ |\varphi_0(x)|\le C |x| \int_{\R^{N-1} }  \frac{|xz|\, U'(|z||x|)}{\abs{\frac{x}{|x|}-z}^{N-2} } \, dz, \]
we deduce
\begin{align*} |\varphi_0(x)| &\le C |x| \int_{\R^{N-1}}  \frac{1 }{ (1+|zx|)^{\gamma+1} \abs{\frac{x}{|x|}-z}^{N-2}} \, dz  \le \frac{C}{|x|^{\gamma}}  \int_{\R^{N-1}}  \frac{1 }{(\frac{1}{|x|}+|z|)^{\gamma+1} \abs{\frac{x}{|x|}-z}^{N-2}} \, dz 
 \end{align*}
which gives the conclusion. 
\end{proof}
\begin{proposition}\label{prop:stima phie}
Let $\gamma > 0$. We have
\[ PU_{\delta, \xi} =U_{\delta, \xi}+ \delta^{-\frac{N}{p+1} +1} \varphi_0\left( \frac{x-\xi}{\delta} \right) + \zeta_\delta, \]
and
\[ PV_{\delta, \xi} = V_{\delta, \xi} +\delta^{-\frac{N}{q+1} +1} \psi_0\left( \frac{x-\xi}{\delta} \right) + s_\delta. \]
where
\[ |\zeta_\delta(x)| \le  \frac{\delta^{-\frac{N}{p+1} +\min \{\gamma, 1\} } (\ln \delta)^\sigma}{(1+|x-\xi|/\delta)^{\gamma}}= \frac{\delta^{-\frac{N}{p+1} +\min \{\gamma, 1\}+\gamma}(\ln \delta)^\sigma}{(\delta+|x-\xi|)^{\gamma}}, \] 
\[ |s_\delta(x)| \le  \frac{\delta^{-\frac{N}{q+1} +1}(\ln \delta)^\tau}{(1+|x-\xi|/\delta)^{N-3}}=  \frac{\delta^{\frac{N}{p+1} }(\ln \delta)^\tau}{(\delta+|x-\xi|)^{N-3}}, \]
and where we denote  $\sigma=1$ if $\gamma=1,2$, $\sigma=0$ otherwise, whereas $\tau=1$ if $N=4, 5$, $\tau=0$ otherwise. Also
\[ \zeta_\delta = O(\delta^{-\frac{N}{p+1} + 1+ \min\{\gamma, 1\}} (\ln \delta)^{\hat \sigma}) \]
\[ s_\delta = O(\delta^{-\frac{N}{q+1} + 2} (\ln \delta)^{\hat \tau}), \]
where $\hat \sigma =1$ if $\gamma=1$, $\hat \sigma =0$ otherwise, and $\hat \tau =1$ if $N=4$, $\hat \tau=0$ otherwise.
\end{proposition}
\begin{proof}
We follow arguments in \cite{ry} and \cite{PSaT}. 
We write
$$  PU_{\delta,\xi} =U_{\delta, \xi}+ \delta^{-\frac{N}{p+1} +1} \varphi_0\left( \frac{x-\xi}{\delta} \right) + \varphi_1 + \varphi_2 + \varphi_3 $$
where
\[ \begin{cases}
- \Delta \varphi_1 + \mu \varphi_1=0 & \text{ in } \Omega \\
\partial_\nu \varphi_1 = - \partial_\nu U_{\delta, \xi} - \delta^{-\frac{N}{p+1} +1} \partial_\nu \left( \varphi_0 \left(\frac{x-\xi}{ \delta} \right) \right)  & \text{ on } \partial \Omega,
\end{cases} 
\] 
\[
\begin{cases}
- \Delta \varphi_2 + \mu \varphi_2=\delta^{-\frac{N}{p+1} + 1} (\Delta - \mu)\left( \varphi_0 \left(\frac{x-\xi}{ \delta} \right) \right) & \text{ in } \Omega \\
\partial_\nu \varphi_2 = 0  & \text{ on } \partial \Omega
\end{cases} 
\]
and
\[
\begin{cases}
- \Delta \varphi_3 + \mu \varphi_3= \mu U_{\delta, \xi} & \text{ in } \Omega \\
\partial_\nu \varphi_3 = 0  & \text{ on } \partial \Omega
\end{cases} \]
We will estimate separately $\varphi_1, \varphi_2, \varphi_3$. We preliminary notice that by translations invariance, we can assume without loss of generality $\xi=0$. 

\textit{Step 1}
Let us start with $\varphi_1$. 

\textit{Step 1.1}
Take $x \in \partial \Omega \cap B_s^c(0)$. Thus
\[ \abs{\partial_\nu U_\delta} \le  \delta^{-\frac{N}{p+1} - 1}  |U'(|x|/\delta)|=
\begin{cases}
O(\delta^{N-2-\frac{N}{p+1}}) & \text{ if } q>\frac{N}{N-2}\\
O(\delta^{q(N-2)-2-\frac{N}{p+1}}) & \text{ if } q<\frac{N}{N-2},  
\end{cases}\]
see also \cite[Equation (B.3)]{PST}. On the other hand, since $|x| \ge s$, 
\[ \abs{ \partial_\nu \varphi_0 \left( \frac x \delta \right) } \le \delta^{-1} \abs{\nabla \varphi_0 \left( \frac x \delta \right) } \le  \frac{C \, \delta^{\gamma}}{(\delta+|x|)^{\gamma+1}}=O(\delta^\gamma)\]
which yields
\begin{equation}\label{stima der fuori} \partial_\nu \varphi_1(x) =\begin{cases}
O(\delta^{N-2-\frac{N}{p+1}}) & \text{ if } q>\frac{N}{N-2}\\
O(\delta^{q(N-2)-2-\frac{N}{p+1}}) & \text{ if } q<\frac{N}{N-2},  
\end{cases} = O(\delta^{-\frac{N}{p+1} + \gamma +1})\end{equation}
if $x \in \partial \Omega \cap B_s^c(0)$.

\textit{Step 1.2}
Let us now consider $x \in \partial \Omega \cap B_s(0)$. Thus, also using
\[  \abs{ U'(|x|/\delta)} \le C (1+|x|/\delta)^{\gamma+2}, \]
one has 
\begin{align} \nonumber  \partial_\nu \ue(x)&= \delta^{-\frac{N}{p+1} - 1}  U'(|x|/\delta) \frac{x}{|x|} \cdot \nu(x)= \delta^{-\frac{N}{p+1} - 1}  U'(|x|/\delta) \frac{ \sum_{j=1}^{N-1} \rho_j x_j^2 + O(|x'|^3)}{|x|\sqrt{1+|\nabla \rho(x')|^2}} \\
\label{stima 1}& =  \delta^{-\frac{N}{p+1} - 1}  U'(|x|/\delta) \frac{ \sum_{j=1}^{N-1} \rho_j x_j^2}{|x|\sqrt{1+|\nabla \rho(x')|^2}} + \delta^{-\frac{N}{p+1}+\gamma+1} O \left( \frac{|x'|^2}{(\delta + |x'|)^{\gamma+2}  } \right) \end{align}
see \cite{PST}. 
Also, (see \cite[A.8--A.13]{PSaT}), let $\hat \varphi_{0,\delta}(x):=\delta^{-\frac{N}{p+1} + 1} \varphi_0 \left( \frac x \delta \right)$.
One has
\begin{align}
\nonumber\partial_\nu \hat \varphi_{0,\delta}(x)&=\nabla \hat \varphi_{0,\delta} (x) \cdot \nu(x) \\
\nonumber &= -\frac{\partial \hat \varphi_{0,\delta}}{\partial x_N} (x', 0) + ((\nabla \hat \varphi_{0,\delta}(x) - \nabla \hat \varphi_{0,\delta}(x', 0))\cdot \nu(x) + \nabla \hat \varphi_{0,\delta} (x', 0) \cdot (\nu(x) + e_N)) \\
\label{stima 2} &= - \delta^{-\frac{N}{p+1} - 1}  U'(|x'|/\delta) \frac{\sum_{j=1}^{N-1} \rho_j x_j^2}{|x'|}  +\delta^{-\frac{N}{p+1}+\gamma+1} O\left( \frac{1}{(\delta +|x'|)^{\gamma}} \right).
\end{align}
Indeed,
\[ \abs{\nabla \hat \varphi_{0,\delta} (x', 0) \cdot (\nu(x) + e_N) } \le \delta^{-\frac{N}{p+1}} \abs{\nabla \varphi_0\left( \frac {x'} \delta, 0\right)} |\nu(x) + e_N|. \]
Using the expression of $\nu$, 
\[ |\nu(x) + e_N| = O(|x'|), \]
whence
\[ \abs{\nabla \hat \varphi_\delta (x', 0) \cdot (\nu(x) + e_N) } \le \delta^{-\frac{N}{p+1}+\gamma+1} O\left( \frac{|x'|}{(\delta +|x'|)^{\gamma+1}} \right). \]
On the other hand, there exist $t_1, \dots t_N \in (0, 1)$ such that 
\begin{align*} \delta^{\frac{N}{p+1}} \abs{\nabla \hat \varphi_\delta(x) - \nabla \hat \varphi_\delta(x', 0)} &= \abs{ \nabla \varphi_0 \left( \frac x \delta \right) -  \nabla \varphi_0 \left( \frac {x'} \delta, 0 \right)} \\
&=  \abs{\left( \frac{\partial^2 \varphi_0}{\partial x_1 \partial x_N} \left( \frac{x'}{\delta}, t_1 \frac{x_1}{\delta}\right), \dots,  \frac{\partial^2 \varphi_0}{\partial^2 x_N} \left( \frac{x'}{\delta}, t_N \frac{x_N}{\delta}\right) \right)} \frac{|x_N|}{\delta}  \\
&\le \frac{C |x_N|}{\delta (1 + |x|/\delta)^{\gamma+2}} = \delta^{\gamma+1} O\left( \frac{|x'|^2}{(\delta+ |x'|)^{\gamma+2}} \right). 
\end{align*}
Observe that if $|x'| $ is close to $0$ by Taylor expansion 
\[  \frac{U'(|x'|/\delta)}{|x'|} = \frac{U'(|x|/\delta)}{|x|\sqrt{1+|\nabla \rho(x')|^2}} + \delta^{\gamma +2} O\left( \frac{1}{(\delta+ |x'|)^{\gamma+2}} \right), \] 
therefore by \eqref{stima 1} and \eqref{stima 2}, 
\[ \partial_\nu \varphi_1(x) = \delta^{-\frac{N}{p+1}+\gamma+1} O\left( \frac{1}{(\delta +|x'|)^{\gamma}} \right) \quad \text{ on } \partial \Omega \cap B_s(0). \]

\textit{Step 1.3}
We are now ready to deduce the estimate on $\varphi_1$. We first recall that there exists a Green function $G$ for the following problem
\[ -\Delta u + \mu u=0 \text{ in } \Omega, \quad \partial_\nu u = f \text{ on } \partial \Omega \]
such that 
\[ u(x)= \int_{\partial \Omega} G(x, y) f(y) \, dy \]
and 
\[ |G(x, y)| \le \frac{C}{|x-y|^{N-2}}. \]
Let 
\[ G_\delta(x, y)=\delta^{N-2} G(\delta x, \delta y). \]
Thus
\[ \hat \varphi_{1, \delta}(x):=\delta^{\frac{N}{p+1}} \varphi_1(\delta x)=\int_{\partial \Omega_\delta} G_\delta(x, y) \partial_\nu \hat \varphi_{1, \delta}(y) \, dy \le C \int_{\partial \Omega_\delta}  \frac{ \partial_\nu \hat \varphi_{1, \delta}(y)}{|x-y|^{N-2}} \, dy. \]
We split
\[ \int_{\partial \Omega_\delta}  \frac{ \partial_\nu \hat \varphi_{1, \delta}(y)}{|x-y|^{N-2}} \, dy = \int_{\partial \Omega_\delta \cap B_s(0)} \frac{ \partial_\nu \hat \varphi_{1, \delta}(y)}{|x-y|^{N-2}} \, dy  + \int_{\partial \Omega_\delta \cap B_s^c(0)} \frac{ \partial_\nu \hat \varphi_{1, \delta}(y)}{|x-y|^{N-2}} \, dy, \]
where $\Omega_\delta:=\delta^{-1} \Omega$, and we estimate the two integrals separately. 
Let us first consider $\partial \Omega_\delta \cap B_s^c(0)$. 
Here, $\partial_\nu \hat \varphi_{1, \delta}(x) = O(\delta^{\gamma+2})$, see \eqref{stima der fuori}. Therefore, 
\begin{equation}\label{stima fuori zeta} \int_{\partial \Omega_\delta \cap B_s^c(0)} \frac{ |\partial_\nu \hat \varphi_{1, \delta}(y)|}{|x-y|^{N-2}} \, dy \le C \delta^{\gamma +2} \int_{\partial \Omega_\delta \cap B_s^c(0)} |x-y|^{2-N} \, dy = O( \delta^{\gamma+1}). \end{equation}

Now let us consider the set $\partial \Omega_\delta \cap B_s(0)$. This time, $\partial_\nu \hat \varphi_{1, \delta} (x) =O(\frac{\delta^2}{(1+|x|)^{\gamma}})$. 
We claim that 
\begin{align}\label{stima gamma} \int_{\partial \Omega_\delta \cap B_s(0)} \frac{ |\partial_\nu \hat \varphi_{1, \delta}(y)|}{|x-y|^{N-2}} \, dy &\le C \delta^{2} \int_{\partial \Omega_\delta \cap B_s(0)} \frac{1}{|x-y|^{N-2}(1+|y|)^\gamma} \le \frac{C \delta^{1+ \min \{\gamma, 1\}} (\ln \delta)^{\tilde \tau}}{(1+|x|)^{\max\{\gamma, 1\}-1}}, \end{align}
where $\tilde \tau=0$ if $\gamma \ne 1$, and $\tilde \tau=1$ if $\gamma=1$. Notice that Lemma A.2 in \cite{PSaT} shows the estimate above in case $\gamma \ge 1$, up to straightening the boundary. 

As for the case $\gamma < 1$, let $x \in \R^N_+$, $d:=\frac 12 |x|$.
The estimate on the set  $B_d(x), B_d(0)$ are exactly as in Lemma A.2 in \cite{PSaT}, and one gets 
\[ \int_{B_d(x)} \frac{1}{|x-y|^{N-2}(1+|y|)^\gamma}  \le C d^{1-\gamma}, \quad \quad  \int_{B_d(0)} \frac{1}{|x-y|^{N-2}(1+|y|)^\gamma}  \le C d^{1-\gamma}. \]
Whereas
\begin{align*} \int_{B_{\frac{1}{\delta}}(0) \setminus (B_d(x) \cup B_d(0))} \frac{1}{|x-y|^{N-2}(1+|y|)^\gamma}  & \le C \, \int_{B_{\frac{1}{\delta}}(0) \setminus  (B_d(x) \cup B_d(0))} |y|^{2-N} (1+|y|)^{-\gamma} \\
& \le C \, \int_d^{\frac{1}{\delta}} r^{-\gamma} \le C\, \delta^{\gamma-1}, \end{align*}
also using $\gamma < 1$. 
Thus we get \eqref{stima gamma}, up to straightening the boundary. Estimate \eqref{stima gamma}, in view of \eqref{stima fuori zeta},  gives 
\[ |\varphi_1(x)| \le \frac{C \delta^{-\frac{N}{p+1} + 1+ \min\{\gamma, 1\}} (\ln \delta)^{\tilde \tau}}{(1+|x|/\delta)^{\max\{ \gamma, 1\}-1}} \le  \frac{C \delta^{-\frac{N}{p+1} + \min \{ \gamma, 1 \}} (\ln \delta)^{\tilde \tau}}{(1+|x|/\delta)^{\gamma}}. \]
Also
\[ \varphi_1 = O(\delta^{-\frac{N}{p+1} + 1+ \min\{\gamma, 1\}} (\ln \delta)^{\tilde \tau}). \]

\textit{Step 2}
We now estimate $\varphi_2$. 
We define $\hat \varphi_{2,\delta}(x):=\delta^{\frac{N}{p+1}} \varphi_{2,\delta}(\delta x)$. Then
\[ -\Delta \hat \varphi_{2,\delta} (x) + \mu \delta^2 \hat \varphi_{2,\delta} = \delta \Delta \varphi_0(x) - \mu \delta^3 \varphi_0(x) \text{ in } \Omega_\delta, \quad \partial_\nu \hat \varphi_{2,\delta}=0 \text{ on } \partial \Omega_\delta. \]
Using Lemma 3.2 in \cite{rewe} and estimates on $\varphi_0$ in Lemma \ref{lem: stima phi} above, we get
\begin{align*} |\hat \varphi_{2,\delta}(x)| &\le C \, \left( \delta \int_{\Omega_\delta} \frac{|\Delta \varphi_0|}{|x-y|^{N-2}} + \delta^3  \int_{\Omega_\delta} \frac{|\varphi_0|}{|x-y|^{N-2}} \right) \\
&= C \, \left( \delta \int_{\Omega_\delta \setminus \R^N_+} \frac{|\Delta \varphi_0|}{|x-y|^{N-2}} + \delta^3  \int_{\Omega_\delta} \frac{|\varphi_0|}{|x-y|^{N-2}} \right) \\
& \le C\, \delta^3 \left( \frac 1 {\delta^2} \int_{\Omega_\delta \setminus \R^N_+} |x-y|^{2-N}(1+|y|)^{-\gamma-2} +  \int_{\Omega_\delta} |x-y|^{2-N}(1+|y|)^{-\gamma} \right). \end{align*}
We claim that 
\begin{equation}\label{claim omega eps} \int_{\Omega_\delta \setminus \R^N_+} |x-y|^{2-N}(1+|y|)^{-\gamma-2} \le C \, \delta^{\gamma}  \frac{1}{(1+|x|)^\gamma}. \end{equation}
Indeed, let us assume first $|x| \le 2$. Then
\[ \int_{\Omega_\delta \setminus \R^N_+} |x-y|^{2-N}(1+|y|)^{-\gamma-2}= \delta^\gamma \int_{\Omega \setminus \R^N_+} |\delta x - y|^{2-N}(\delta+|y|)^{-\gamma -2} \le C \, \delta^\gamma, \]
since $|x\delta|\le 2 \delta < \frac 12 d(\Omega \setminus \R^N_+, 0)$ if $\delta$ is small enough, and $|y|> d(\Omega \setminus \R^N_+, 0)$.
On the other hand, let $|x| \ge 2$. Then $d:=\frac 12 |x| \ge 1$, and there exist two positive constants $c_1>c_2$ such that 
\[ \Omega_\delta \setminus \R^N_+ \subseteq B_{\frac{c_1d}{\delta}}(0) \setminus B_{\frac{c_2d}{\delta}}(0). \]
Thus
\begin{multline*}  \int_{\Omega_\delta \setminus \R^N_+} |x-y|^{2-N}(1+|y|)^{-\gamma-2} \\ \le \int_{B_{\frac{d}{\delta}(x)}} |x-y|^{2-N}(1+|y|)^{-\gamma-2} + \int_{B_{\frac{c_1d}{\delta}}(0) \setminus (B_{\frac{c_2d}{\delta}}(0) \cup B_{\frac{d}{\delta}}(x))} |x-y|^{2-N}(1+|y|)^{-\gamma-2}. \end{multline*}
The first integral gives 
\begin{equation}\label{first claim} \int_{B_{\frac{d}{\delta}(x)}} |x-y|^{2-N}(1+|y|)^{-\gamma-2} \le C\,  \frac{\delta^{\gamma+2}}{d^{\gamma+2}} \int_{B_{\frac{d}{\delta}(0)}} |y|^{2-N} \le C\, \frac{\delta^\gamma}{d^\gamma}. \end{equation}
As for the second one, we preliminary notice that for any $y \in B_{\frac{d}{\delta}}^c(x)$ one has $|x-y| \ge d/\delta$. Thus if $|y| \le 2 |x|$ then
\[ |x-y|^{N-2} \ge \left( \frac d \delta \right)^{N-2} \ge C |x|^{N-2} \ge C |y|^{N-2}. \]
Moreover, if $|y| \ge 2|x|$ then $|x-y| \ge \frac12\, |y |$, and $|x-y|^{N-2} \ge C |y|^{N-2}$. As a consequence, 
\begin{align}\label{second claim} \nonumber \int_{B_{\frac{c_1d}{\delta}}(0) \setminus (B_{\frac{c_2d}{\delta}}(0) \cup B_{\frac{d}{\delta}}(x))} |x-y|^{2-N}(1+|y|)^{-\gamma-2} & \le C\, \int_{B_{\frac{c_1d}{\delta}}(0) \setminus (B_{\frac{c_2d}{\delta}}(0) \cup B_{\frac{d}{\delta}}(x))}  |y|^{2-N}(1+|y|)^{-2-\gamma} \\
& \le C\, \int_{\frac{c_2d}{\delta}}^{\frac{c_1d}{\delta}} r^{-1-\gamma} 
\le C\, \frac{\delta^\gamma}{d^\gamma}. \end{align}
Estimate \eqref{claim omega eps} now follows from \eqref{first claim} and \eqref{second claim}. 
Finally, by similar arguments as the ones we used for \eqref{stima gamma}, 
\[ \int_{\Omega_\delta} |x-y|^{2-N}(1+|y|)^{-\gamma}  \le C \frac{\delta^{\min\{\gamma, 2\} -2} (\ln \delta)^{\tilde \sigma}}{(1+|x|)^{\max\{\gamma, 2\}-2}} \]
where $\tilde \sigma=1$ if $\gamma=2$ and $\tilde \sigma=0$ if $\gamma \ne 2$. 

Now, we conclude 
\[  |\hat \varphi_{2,\delta}(x)| \le C \, \delta^3 \left( \frac{\delta^{\gamma-2}}{(1+|x|)^\gamma} +  \frac{\delta^{\min\{\gamma, 2\} -2} (\ln \delta)^{\tilde \sigma}}{(1+|x|)^{\max\{\gamma, 2\}-2}}, \right)= C \, \delta \left( \frac{\delta^{\gamma}}{(1+|x|)^\gamma} +  \frac{\delta^{\min\{\gamma, 2\} } (\ln \delta)^{\tilde \sigma}}{(1+|x|)^{\max\{\gamma, 2\}-2}}, \right)  \]

Therefore, 
\[  |\varphi_2(x)| \le \frac{C\,  \delta^{-\frac{N}{p+1} + 1}(\ln \delta)^{\tilde \sigma }}{(1+|x|/\delta)^\gamma} \]
and
\[ \varphi_2=O(\delta^{-\frac{N}{p+1} + 1+ \min\{\gamma, 2 \}} (\ln \delta)^{\tilde \sigma }) \]
and the conclusion follows.

\textit{Step 3.} We are left with the estimate of $\varphi_3$.
Take $\hat \varphi_{3,\delta}(x):=\delta^{\frac{N}{p+1}} \varphi_3(\delta x)$. Then
\[ \begin{cases}
-\Delta \hat \varphi_{3,\delta} + \mu \delta^2 \hat \varphi_{3,\delta}= \mu \delta^2 U & \text{ in } \Omega_\delta \\
\partial_\nu \hat \varphi_{3,\delta}=0 & \text{ on } \partial \Omega_\delta. 
\end{cases} \]
Hence
\[ |\hat \varphi_{3,\delta} (x)| \le C \delta^2 \int_{\Omega_\delta} \frac{U}{|x-y|^{N-2}} \le C \delta^2 \int_{\Omega_\delta} \frac{1}{(1+|y|)^{\gamma+1}|x-y|^{N-2}}. \]
As above, we conclude
\[ |\hat \varphi_{3,\delta}(x)| \le C \frac{\delta^{1+\min\{\gamma, 1 \}} (\ln \delta)^{\tilde \tau}}{(1+|x|)^{\max\{\gamma, 1\} - 1}} \]
where $\tilde \tau=1$ if $\gamma=1$, $\tilde \tau=0$ otherwise.
Whence
\[ \varphi_3 = O(\delta^{-\frac{N}{p+1} + 1+ \min\{\gamma, 1\}} (\ln \delta)^{\tilde \tau}) \]
and 
\[ |\varphi_3(x)| \le C \frac{ \delta^{-\frac{N}{p+1} + \min \{ \gamma, 1 \}} (\ln \delta)^{\tilde \tau}}{(1+|x|/\delta)^{\gamma}}. \qedhere \]
\end{proof}
\begin{corollary}\label{cor:stima phie}
Let $\gamma \ge 1$. We have
\[ PU_{\delta, \xi} =U_{\delta, \xi}+ \delta^{-\frac{N}{p+1} +1} \varphi_0\left( \frac{x-\xi}{\delta} \right) + \zeta_\delta, \]
and
\[ PV_{\delta, \xi} = V_{\delta, \xi} +\delta^{-\frac{N}{q+1} +1} \psi_0\left( \frac{x-\xi}{\delta} \right) + s_\delta. \]
where
\[ |\zeta_\delta(x)| \le  \frac{\delta^{-\frac{N}{p+1} +1 } (\ln \delta)^\sigma}{(1+|x-\xi|/\delta)^{\gamma}}= \frac{\delta^{-\frac{N}{p+1} +1+\gamma}(\ln \delta)^\sigma}{(\delta+|x-\xi|)^{\gamma}}, \] 
\[ |s_\delta(x)| \le  \frac{\delta^{-\frac{N}{q+1} +1}(\ln \delta)^\tau}{(1+|x-\xi|/\delta)^{N-3}}=  \frac{\delta^{\frac{N}{p+1} }(\ln \delta)^\tau}{(\delta+|x-\xi|)^{N-3}}, \]
and where we denote  $\sigma=1$ if $\gamma=1,2$, $\sigma=0$ otherwise, whereas $\tau=1$ if $N=4, 5$, $\tau=0$ otherwise. Also
\[ \zeta_\delta = O(\delta^{-\frac{N}{p+1} + 2} (\ln \delta)^{\hat \sigma}) \]
\[ s_\delta = O(\delta^{-\frac{N}{q+1} + 2} (\ln \delta)^{\hat \tau}), \]
where $\hat \sigma =1$ if $\gamma=1$, $\hat \sigma =0$ otherwise, and $\hat \tau =1$ if $N=4$, $\hat \tau=0$ otherwise.
\end{corollary} 

By completely similar arguments as above, we conclude the following 
\begin{proposition}\label{prop:expansion pphi}
Let $\gamma > 0$. 
We have
\[ P\Phi^i_{\delta, \xi} =\Phi^i_{\delta, \xi}+ \partial_i\left(\delta^{-\frac{N}{p+1} +1} \varphi_0\left( \frac{x-\xi}{\delta} \right) \right)+ \tilde \zeta_\delta, \]
and
\[ P\Psi^i_{\delta, \xi} = \Psi^i_{\delta, \xi} +\partial_i \left(\delta^{-\frac{N}{q+1} +1} \psi_0\left( \frac{x-\xi}{\delta} \right) \right)+ \tilde s_\delta. \]
where
\begin{align*} |\tilde \zeta_\delta(x)| &\le  
\begin{cases}
C\, \frac{\delta^{-\frac{N}{p+1} + \min \{\gamma, 1\}  - 1} (\ln \delta)^\sigma}{(1+|x-\xi|/\delta)^{\gamma}}, & \text{ if $i=0$}\\
C\, \frac{\delta^{-\frac{N}{p+1}+ \min \{\gamma, 1\}  - 1 } (\ln \delta)^\sigma}{(1+|x-\xi|/\delta)^{\gamma+1}}, & \text{ if $i\ne 0$,}
\end{cases} \end{align*}
and 
\begin{align*} |\tilde s_\delta(x)| &\le 
\begin{cases}   C\, \frac{\delta^{-\frac{N}{q+1}}(\ln \delta)^\tau}{(1+|x-\xi|/\delta)^{N-3}} & \text{ if $i=0$}\\
C\, \frac{\delta^{-\frac{N}{q+1} }(\ln \delta)^\tau}{(1+|x-\xi|/\delta)^{N-2}} & \text{ if $i\ne 0$}
 \end{cases} \end{align*}
where $\sigma=1$ if $\gamma=1,2$, $\sigma=0$ otherwise; whereas $\tau=1$ if $N=4, 5$, $\tau=0$ otherwise. 
Moreover,
\[ \tilde \zeta_\delta = O(\delta^{-\frac{N}{p+1} + \min \{\gamma, 1\}} (\ln \delta)^{\hat \sigma}) \]
\[ \tilde s_\delta = O(\delta^{-\frac{N}{q+1} + 1} (\ln \delta)^{\hat \tau}), \]
where $\hat \sigma =1$ if $\gamma=1$, $\hat \sigma =0$ otherwise, and $\hat \tau =1$ if $N=4$, $\hat \tau=0$ otherwise.

\end{proposition}

\section*{Acknowledgments.}

Delia Schiera is partially supported by the Portuguese government through FCT - Funda\c c\~ao para a Ci\^encia e a Tecnologia, I.P., under the projects UID/MAT/04459/2020 and PTDC/MAT-PUR/1788/2020, and under the Scientific Employment Stimulus - Individual Call (CEEC Individual), https://doi.org/10.54499/2020.02540.CEECIND/CP1587/CT0008. Angela Pistoia acknowledges support of 
INDAM-GNAMPA project ``Problemi di doppia curvatura su variet\`a  a bordo e legami con le EDP di tipo ellittico'' and of the project 
``Pattern formation in nonlinear phenomena'' funded by the MUR Progetti di Ricerca di Rilevante Interesse Nazionale (PRIN) Bando 2022 grant 
20227HX33Z.

\bibliography{biblio}
\bibliographystyle{abbrv}

\bigskip
\noindent \textbf{Angela Pistoia}\\
Dipartimento di Scienze di Base e Applicate per l’Ingegneria\\
Sapienza Università di Roma\\
Via Scarpa 16, 00161 Roma, Italy\\
\texttt{angela.pistoia@uniroma1.it} 
\vspace{.3cm}

\noindent \textbf{Delia Schiera}\\
Departamento de Matemática do Instituto Superior Técnico\\
Universidade de Lisboa\\
Av. Rovisco Pais\\
1049-001 Lisboa, Portugal\\
\texttt{delia.schiera@tecnico.ulisboa.pt}

\end{document}